\documentclass{article}

\usepackage{amsthm,amsmath,amssymb}
\usepackage[numbers]{natbib}
\usepackage{marvosym}

\newtheorem{theorem}{Theorem}
\newtheorem{proposition}[theorem]{Proposition}
\newtheorem{lemma}[theorem]{Lemma}

\newtheorem{remark}[theorem]{Remark}
\newtheorem{example}[theorem]{Example}

\newcommand{\E}{{\mathbb E}}
\newcommand{\R}{{\mathbb R}}
\renewcommand{\P}{{\mathbb P}}
\newcommand{\Acal}{{\cal A}}

\newcounter{rcnt}[section]
\renewcommand{\thercnt}{(\roman{rcnt})}

\begin{document}

\title{Concentration of the spectral measure of large Wishart matrices
	with dependent entries}

\author{Adityanand Guntuboyina and Hannes Leeb\\	
(Yale University)}

\date{October 2008}

\maketitle

\begin{abstract}
We derive concentration inequalities for the spectral measure of
large random matrices, allowing for certain forms of dependence.
Our main focus is on empirical covariance (Wishart) matrices, but
general symmetric random matrices are also considered.
\end{abstract}

\section{Introduction}

In this short paper, we study concentration of the spectral measure of
large random matrices whose elements need not be independent.
In particular, we derive a concentration inequality for Wishart matrices
of the form $X'X/m$ in the important setting where the rows of the $m \times n$ matrix $X$ are
independent but the elements within each row may depend on each other; see
Theorem~\ref{t1}.
We also obtain similar results for
other random matrices with dependent entries; see
Theorem~\ref{t2}, Theorem~\ref{t3}, and the attending examples, which include 
a random graph with dependent edges, and vector time series.

Large random matrices have been the focus of intense research in recent
years; see Bai~\citep{Bai99a} and Guionnet~\citep{Gui04a} for surveys. 
While most of this literature deals with
the case where the underlying matrix has independent entries, 
comparatively little is known for dependent cases.
G\"otze and Tikhomirov~\citep{Got04a}  show that the 
expected 
spectral distribution of an empirical covariance matrix  $X'X/m$  
converges to the Mar\v{c}enko-Pastur law under
conditions that allow for some form of dependence among the
entries of $X$.
Bai and Zhou~\citep{Bai08a} analyzed the limiting spectral distribution
of $X'X/m$ when the row-vectors of $X$
are independent (allowing for certain forms of 
dependence within the row-vectors of $X$).
Mendelson and Pajor~\citep{Men06a} considered
$X'X/m$ in the case where the row-vectors of $X$ are 
independent and identically distributed (i.i.d.); under some additional
assumptions, they derive a concentration result for the operator norm
of $X'X/m - E(X'X/m)$. 
Boutet de Monvel and Khorunzhy~\citep{Bou98a} studied the
limiting behavior of the spectral distribution and of the operator
norm of symmetric Gaussian matrices with dependent entries.

For large random matrices similar to those considered here,
concentration of the spectral measure is also studied
by Guionnet and Zeitouni~\citep{Gui00a}, who consider Wishart
matrices $X'X/m$ where
the entries $X_{i,j}$ of $X$ are independent, as well as
Hermitian matrices with independent entries on and above the diagonal,
and by 
Houdre and Xu~\citep{Hou08a}, who obtained
concentration results for random matrices with stable entries,
thus allowing for certain forms of dependence.
For matrices with dependent entries,
we find that concentration of the spectral measure can be less pronounced
than in the independent case.
Technically, our results rely on 
a slight extension of a result of Talagrand~\citep{Tal96b}, and
on McDiarmid's bounded difference inequality~\citep{Mcd89a}.

\section{Results}

Throughout, the eigenvalues of a symmetric $n\times n$ matrix $M$
are denoted by $\lambda_1(M) \leq \dots \leq \lambda_n(M)$, and we
write $F_M(\lambda)$ for the cumulative distribution function (c.d.f.) of the
spectral distribution of $M$, i.e., 
$F_M(\lambda) = n^{-1} \sum_{i=1}^n \{ \lambda_i(M) \leq \lambda\}$, 
$\lambda \in \R$. The integral of a function $f(\cdot)$ with respect to the 
measure induced by $F_M$ is denoted by $F_M(f)$, i.e., 
$$
F_M(f) \quad=\quad  \frac{1}{n} \sum_{i=1}^n f(\lambda_i(M)).
$$
For certain classes of random matrices $M$
and certain classes of functions $f$, we will show that $F_M(f)$ 
is concentrated around its expectation 
$\E F_M(f)$ or around any median
$\text{med } F_M(f)$. For a Lipschitz function $g$, we write $||g||_L$ for its Lipschitz constant. Moreover, we also consider functions $f:(a, b)\to\R$ that are of bounded variation on $(a, b)$ (where $-\infty \leq a < b \leq \infty$), in the sense that 
\begin{equation}\nonumber
V_f(a, b) \qquad = \qquad 
\sup_{n \geq 1} \;\; \sup_{a < x_0 \leq x_1 \leq \dots \leq x_n < b} \;\;\sum_{k = 1}^{n} 
|f(x_k) - f(x_{k-1})| \
\end{equation}
is finite; cf. Section X.1 in~\citep{Lan93a}. [A function $f$ is of bounded variation on $(a, b)$ if and only if it can be written as the difference of two bounded monotone functions on $(a, b)$, as is easy to see. Note that the indicator function  $g: x \mapsto \{ x\leq \lambda\}$ is of bounded variation on $\R$ with $V_g(\R) = 1$ for each $\lambda \in \R$.]

The following result establishes concentration of $F_S(f)$ for 
Wishart matrices $S$ of the form  $S = X'X/m$
where we only require that the rows of $X$ are independent 
(while allowing for dependence {\em within} each row of $X$).
See also Example~\ref{ma} and Example~\ref{gzgen}, which follow,
for scenarios that also allow for some dependence among the rows of $X$.

\begin{theorem}\label{t1}
Let $X$ be an $m\times n$ matrix whose row-vectors are independent,
set $S=X'X/m$, and fix $f:\R \to\R$.
\begin{list}{\thercnt}{\usecounter{rcnt}\setlength{\leftmargin}{5pt}}
\item \label{t1.i}
Suppose that $f$ is such that the mapping $x\mapsto f(x^2)$ is convex
and Lipschitz,
and suppose that $|X_{i,j}|\leq 1$ for each $i$ and $j$.
For each $\epsilon>0$, we then have
\begin{equation}\label{t1.i.1}
\P \left( \left| F_S(f) -  \text{med } F_S(f) \right|  \geq \epsilon \right)
\quad\leq\quad 
4 \exp\left[
	-\frac{ n m }{ n+m} \,\frac{ \epsilon^2}{8 ||f(\cdot^2)||_L^2}
\right].
\end{equation}
[From the upper bound \eqref{t1.i.1} one can also obtain a similar bound
for $\P(|F_S(f) - \E F_S(f)| \geq \epsilon)$ using standard methods.]
\item \label{t1.ii}
Suppose that $f$ is of bounded variation on $\R$. For each $\epsilon>0$, we then have
\begin{equation}\label{t1.ii.1}
\P \left( \left| F_S(f) - \E F_S(f) \right| \geq \epsilon \right)
\quad\leq\quad 2 \exp \left[ -\frac{n^2}{m} \frac{ 2\epsilon^2}{V_f^2(\R)}
\right].
\end{equation}
In particular, for each $\lambda\in\R$ and each $\epsilon>0$,
the probability $\P(|F_S(\lambda) - \E F_S(\lambda)| \geq \epsilon)$ is bounded by
the right-hand side of \eqref{t1.ii.1} with  $V_f(\R)$ replaced by $1$.
\end{list}
\end{theorem}

The upper bounds in Theorem~\ref{t1} are of the form
\begin{equation}
\label{bnd}
\P( |F_S(f) - A | \geq \epsilon) \quad\leq \quad B \exp\left[ -n C\right],
\end{equation}
where $A$, $B$, and $C$ equal $\text{med }F_S(f)$, $4$,
and $m \epsilon^2 / \left( (n+m) 8 ||f(\cdot^2)||_L^2 \right)$ in part~\ref{t1.i}
and $\E F_S(f)$, $2$, and $n 2 \epsilon^2 / (m V_f^2 )$ in part~\ref{t1.ii},
respectively.
For the interesting case where $n$ and $m$ both go to infinity at the
same rate,
the next example shows that these bounds can not be
improved qualitatively without imposing additional assumptions.

\begin{example}
\label{ex1}
Let $n=m=2^k$, and let $X$ be the $n\times n$ matrix whose
$i$-th row is $R_i v_i'$, where $R_1,\dots, R_n$ are i.i.d. with $\P(R_1 = 0) = \P(R_1 = 1) = 1/2$, and where $v_1,\dots,v_n$ are orthogonal $n$-vectors with $v_i \in \{-1,1\}^n$ for each $i$. [The $v_i$'s can be
obtained, say, from the first $n$ binary Walsh functions; cf.~\citep{Wal23a}.]
Note that the eigenvalues of $S = X'X/m$ are $R_1^2,\dots, R_n^2$.
Set $f(x) = x$ for $x\in \{0,1\}$.
Then
$n F_S(f)$ is binomial distributed with parameters $n$ and $1/2$, i.e.,
$n F_S(f) \sim B(n,1/2)$.
By Chernoff's method (cf. Theorem~1 of~\citep{Che52a}),
we hence obtain that
\begin{equation}\label{ex1.1}
\P(F_S(f) -\E F_S(f) \geq \epsilon)  \quad=\quad \exp\left[ -n (C(\epsilon) + o(1)) \right],
\end{equation}
for $0 < \epsilon < 1/2$ and
as $n\to\infty$ with $k\to\infty$, where here $C(\epsilon)$ equals 
$\log(2)+(1/2+\epsilon)\log(1/2+\epsilon) + (1/2-\epsilon) \log(1/2-\epsilon)$;
the same is true if
$\E F_S(f)-F_S(f)$ replaces $F_S(f)-\E F_S(f)$ in \eqref{ex1.1}.
These statements continue to hold with $\text{med } F_S(f)$ replacing
$\E F_S(f)$, because the mean coincides with the median here.
To apply Theorem~\ref{t1}\ref{t1.i}, we extend $f$ by setting
$f(x) = \sqrt{|x|}$ for $x\in \R$; to apply Theorem~\ref{t1}\ref{t1.ii}, 
extend $f$ as $f(x) = 1\{x\leq 1/2\}$. Theorem~\ref{t1}\ref{t1.i} and 
Theorem~\ref{t1}\ref{t1.ii} give us that the left hand side 
of~\eqref{ex1.1} is bounded by terms of the form 
$4 \exp\left[ -n C_1(\epsilon)\right]$ and 
$2 \exp\left[ -n C_2(\epsilon)\right]$, 
respectively, for some functions $C_1$ and $C_2$ of $\epsilon$. It 
is easy to check that $C(\epsilon)/C_i(\epsilon)$ is increasing in 
$\epsilon$ for $i = 1, 2$,
and that
\begin{equation*}
\lim_{\epsilon \downarrow 0} \frac{C(\epsilon)}{C_1(\epsilon)} = 
32 \qquad \text{ and } \qquad \lim_{\epsilon \downarrow 0} 
\frac{C(\epsilon)}{C_2(\epsilon)} = 1.
\end{equation*}
Hence, both parts of Theorem~\ref{t1} give upper bounds 
with the correct rate ($-n$) in the exponent. 
The constants $C_i(\epsilon)$, $i=1,2$, both are sub-optimal, i.e., they
are too small,
but the constant
$C_2(\epsilon)$, which is obtained from 
Theorem~\ref{t1}\ref{t1.ii}, is close to the optimal constant
for small $\epsilon$.
\end{example}

Under additional assumptions on the law of $X$, 
$F_S(f)$ can concentrate faster than indicated by \eqref{bnd}. 
In particular, 
in the setting of Theorem~\ref{t1}\ref{t1.i} and
for the case where
all the elements $X_{i,j}$ of $X$ are independent,
Guionnet and  Zeitouni~\citep{Gui00a}
obtained bounds of the same form as \eqref{bnd} but with $n^2$ 
replacing $n$ in the exponent, for functions $f$ such that 
$x \mapsto f(x^2)$ is convex and Lipschitz. 
(This should be compared with Example~\ref{gzgen} below.)
However, if 
$f$ does not satisfy this requirement, but is of bounded variation on 
$\R$ so that Theorem~\ref{t1}\ref{t1.ii} applies, then the upper bound 
in~\eqref{t1.ii.1} can not be improved qualitatively without additional 
assumptions, even in the case when all the elements $X_{i,j}$ of $X$ 
are independent. This is demonstrated by the following example.

\begin{example}
\label{ex2}
Let $X$ be the $n\times n$ diagonal matrix $\text{diag}(R_1,\dots, R_n)$, 
where $R_1,\dots, R_n$ are as in Example~\ref{ex1}. Set $f(x) = 1\{x\leq 0\}$. Clearly, Theorem~\ref{t1}\ref{t1.ii} applies here so that the left hand side~\eqref{t1.ii.1} is bounded by $2 \exp\left[ -n C_2(\epsilon)\right]$ for $C_2(\epsilon)$ as in Example~\ref{ex1}. Moreover, since for each $i$, $f(R_i^2/n) = 1 - R_i$, it follows that $n F_S(f) \sim B(n,1/2)$, and then \eqref{ex1.1} holds again.
\end{example}

Theorem~\ref{t1} can also be used to get concentration inequalities for the empirical distribution of the singular values of a non-symmetric $n \times m$ matrix $X$ with independent rows. Indeed, the $i$-th singular value of $X/\sqrt{m}$ is just the square root of the $i$-th eigenvalue of $X'X/m$.

Both parts of
Theorem~\ref{t1} are in fact special cases of more general results
that are presented next. The following two theorems, the first of which should be compared with Theorem~1.1(a) of \citep{Gui00a},  apply
to a variety of random matrices besides those considered in Theorem~\ref{t1};
some examples are given later in this section. 

\begin{theorem}\label{t2}
Let $M$ be a random symmetric $n\times n$ matrix that is a function of
$m$ independent $[-1,1]^p$-valued random vectors $Y_1,\dots, Y_m$ i.e., $M=M(Y_1,\dots, Y_m)$.
Assume that $M(\cdot)$ is linear and Lipschitz with Lipschitz constant 
$C_M$
when considered as a function from $[-1,1]^{m p}$ with the Euclidean norm
to the set of all symmetric $n\times n$ matrices with the Euclidean norm on $\R^{n(n+1)/2}$ (we view symmetric $n \times n$ matrices as elements of $\R^{n(n+1)/2}$ by collecting the entries on and above the diagonal).
Finally, assume that
$f:\R\to\R$ is convex and Lipschitz with Lipschitz constant $||f||_L$.
For $S=M/\sqrt{m}$, we then have
\begin{equation}
\label{t2.1}
\P\left(
\left| 
F_{S}(f) - \text{med }F_{S}(f)
\right|\,\geq\,\epsilon
\right)
\quad\leq\quad
4 \exp\left[
	- \frac{n m }{p}					
	\frac{\epsilon^2}{32 C_M^2 ||f||_L^2}
\right]
\end{equation}
for each $\epsilon>0$.
\end{theorem}

\begin{theorem}\label{t3}
Let $M$ be a random symmetric $n\times n$ matrix that is a function of
$m$ independent random quantities $Y_1,\dots, Y_m$, i.e.,
$M = M(Y_1,\dots, Y_m)$. Write $M_{(i)}$ for the matrix obtained from $M$
after replacing $Y_i$ by an independent copy, i.e.,
$M_{(i)} = M(Y_1,\dots, Y_{i-1}, Y_i^\ast, Y_{i+1},\dots,Y_m)$
where $Y_i^\ast$ is distributed as $Y_i$ and independent of
$Y_1,\dots, Y_m$ $(i=1,\dots,m)$.
For $S=M/\sqrt{m}$ and $S_{(i)} = M_{(i)}/\sqrt{m}$, assume that
\begin{equation}\label{t3.1}
||F_S - F_{S_{(i)}}||_\infty \quad\leq \quad r/n
\end{equation}
holds (almost surely) 
for each $i=1,\dots,m$ and for some (fixed) integer $r$. Finally,
assume that $f:\R\to\R$ is of bounded variation on $\R$.
For each $\epsilon>0$, we then have
\begin{equation}\label{t3.2}
\P \left( \left| F_S(f) - \E F_S(f) \right| \, \geq \, \epsilon \right)
\quad\leq\quad 2 \exp \left[ -\frac{n^2}{m}  
	\frac{ 2\epsilon^2}{r^2 V_f^2(\R)}
\right].
\end{equation}
Also, if $a$ and $b$, $-\infty \leq a < b \leq \infty$, are such that 
$\P a < \lambda_1(S) \text{ and } \lambda_n(S) < b  = 1$, 
then~\eqref{t3.2} holds for each function $f : (a, b) \rightarrow \R$ of 
bounded variation on $(a, b)$, where now $V_f(a, b)$ replaces $V_f(\R)$ on 
the right hand side of~\eqref{t3.2}. 
\end{theorem}

To apply Theorem~\ref{t3}, one needs to establish the inequality in \eqref{t3.1}
for each $i=1,\dots, m$. This can often be accomplished by using the following
lemma, which is taken from Bai~\citep{Bai99a}, Lemma 2.2 and~2.6,
and which is a simple consequence of the interlacing theorem.
[Consider
	a symmetric $n\times n$ matrix $A$ and denote
	its $(n-1)\times (n-1)$ major submatrix by $B$. The interlacing
	theorem, a direct consequence of the Courant-Fisher formula,
	states that $\lambda_i(A) \leq \lambda_i(B) \leq \lambda_{i+1}(A)$
	for $i=1,\dots,n-1$.]

\begin{lemma}\label{bai}
Let $A$ and $B$ be symmetric $n \times n$ matrices and let $X$ and $Y$ be $m \times n$ matrices. Then the following inequalities hold:
\begin{equation*}
||F_A - F_B||_{\infty} \quad \leq \quad \frac{\text{rank}(A - B)}{n},
\end{equation*}
and
\begin{equation*}
||F_{X'X} - F_{Y'Y}||_{\infty} \quad \leq \quad \frac{\text{rank}(X - Y)}{n}.
\end{equation*}
\end{lemma}

We now give some examples where Theorem~\ref{t2} or 
Theorem~\ref{t3} can be applied, the latter with the help of Lemma~\ref{bai}.

\begin{example} \label{ex3}
Consider a network of, say, social connections or relations between
a group of $n$ entities that enter the  group sequentially and that establish
connections to group members that entered before as follows:
For the $i$-th entity that enters the group,
connections to the existing group members, labeled 
$1,\dots, i-1$, are chosen according to some probability distribution,
independently of the choices made by all the other entities.
Denote the $n\times n$ adjacency matrix of the
resulting random graph by $M$, 
and write $Y_i$ for the $n$-vector $(M_{i,1}, M_{i,2}, \dots, M_{i,i}, 0, \dots, 0)'$ for $i = 1, \dots, n$. By construction, $Y_1,\dots, Y_n$ are independent and $M$ (when considered as a function of $Y_1, \dots, Y_n$ as in Theorem~\ref{t2}) is linear and Lipschitz with Lipschitz constant 1. Hence Theorem~\ref{t2} is applicable with $m = p = n$ and $C_M = 1$. 

Theorem~\ref{t3} can also be applied here. To check condition~\eqref{t3.1},  write $M_{(i)}$ for the matrix obtained from $M$ by 
replacing $Y_i$ by an independent copy denoted by $Y_i^\ast$ as in Theorem~\ref{t3}.
Clearly, the $i$-th row of the matrix $M-M_{(i)}$ 
equals 
$\delta_i = (Y_{i,1} - Y_{i,1}^\ast, \dots Y_{i,i} - Y_{i,i}^\ast, 0,\dots, 0)$,
the $i$-th column of $M-M_{(i)}$ equals $\delta_i'$, and the
remaining elements of $M-M_{(i)}$ all equal zero.
Therefore, the rank of $M-M_{(i)}$ is at most two. Using Lemma~\ref{bai}, we see that Theorem~\ref{t3} is applicable here with $r = 2$ and $m = n$.
\end{example}

The following two examples deal with the sample covariance matrix of vector moving average (MA) processes. For the sake of simplicity, we only consider MA processes of order 2. Our arguments can be extended to also handle MA processes  of any fixed and finite order. In Example~\ref{ma}, we consider an MA(2) process with independent innovations, allowing for arbitrary dependence within each innovation, and obtain concentration inequalities of the form~\eqref{bnd}. In Example~\ref{gzgen}, we consider the case where each innovation has independent components (up to a linear function) and obtain a concentration inequality of the form~\eqref{bnd} but with $n^2$ replacing $n$ in the exponent.

\begin{example}\label{ma}
Consider an $m \times n$ matrix $X$ whose row-vectors follow a 
vector MA process of order 2 i.e., $(X_{i,.})' = Y_{i+1} + B Y_i$ 
for $i = 1 \dots m$, where $Y_1, \dots Y_{m+1}$ are $m+1$ independent 
$n$-vectors and $B$ is some fixed $n \times n$ matrix. Set $S = X'X/m$.
\begin{list}{\thercnt}{\usecounter{rcnt}\setlength{\leftmargin}{5pt}}
\item \label{ma.i}
Suppose that $f$ is such that the mapping $x\mapsto f(x^2)$ is convex
and Lipschitz,
and suppose that $Y_i \in \left[ -1, 1 \right]^n$ for each 
$i = 1, \dots, m + 1$.
For each $\epsilon>0$, we have
\begin{equation}\label{ma.i.eq}
\P \left( \left| F_S(f) -  \text{med } F_S(f) \right|  \geq \epsilon \right)
\quad\leq\quad 
4 \exp\left[
-\frac{ n m^2 }{ (m+1)(n+m)} \,\frac{ \epsilon^2}{8 C_B^2 ||f(\cdot^2)||_L^2}
\right].
\end{equation}
Here $C_B$ equals $1 + ||B||$, where $||B||$ is the operator norm of the 
matrix $B$.
\item \label{ma.ii}
Suppose that $f$ is of bounded variation on $\R$. For each $\epsilon>0$, we 
then have
\begin{equation}\label{ma.ii.eq}
\P \left( \left| F_S(f) - \E F_S(f) \right| \geq \epsilon \right)
\quad\leq\quad 2 \exp \left[ -\frac{n^2}{m + 1} \frac{ \epsilon^2}{2 V_f^2(\R)}
\right].
\end{equation}
\end{list}
The proofs of~\eqref{ma.i.eq} and~\eqref{ma.ii.eq} follow essentially the 
same argument as used in the proof of Theorem~\ref{t1}  
using the particular structure of the matrix $X$ as considered here. 
\end{example}

\begin{example}\label{gzgen}
As in Example~\ref{ma}, consider an $m\times n$ matrix $X$
whose row-vectors follow a vector MA(2) process 
$(X_{i,\cdot})' = Y_{i+1} + B Y_i$  for some fixed $n\times n$ matrix $B$,
$i=1,\dots,m$. For the innovations $Y_i$, we now assume
that $Y_i = U Z_i$, where  $U$ is a fixed $n\times n$ matrix, and where
the $Z_{i,j}$, $i=1,\dots, m+1$, $j=1,\dots,n$, are independent
and satisfy $|Z_{i,j}| \leq 1$.  Set $S = X'X/m$.
For a function $f$
such that the mapping $x\mapsto f(x^2)$ is convex and Lipschitz,
we then obtain that
\begin{equation}\label{gzgen.eq}
\P \left( \left| F_S(f) -  \text{med } F_S(f) \right|  \geq \epsilon \right)
\quad\leq\quad 
4 \exp\left[
	-\frac{ n^2 m }{ n+m} \,\frac{ \epsilon^2}{8 C^2 ||f(\cdot^2)||_L^2}
\right]
\end{equation}
for each $\epsilon>0$, where $C$ is shorthand for $C=(1+||B||) \,||U||$
with $||B||$ and $||U||$ denoting the operator norms of the indicated matrices.
The relation 
\eqref{gzgen.eq} is derived by essentially repeating
the proof of Theorem~\ref{t1}\ref{t1.i} and by employing the particular 
structure of the matrix $X$ as considered here.

We note that the statement in the previous paragraph reduces to
Corollary 1.8(a) in~\citep{Gui00a} if one sets $B$ to the zero matrix
and $U$ to the identity matrix.
Moreover, we 
note that Theorem~\ref{t3} can also be applied here (similarly to
Example~\ref{ma}\ref{ma.ii}), but the resulting upper bound does
not improve upon \eqref{ma.ii.eq}.
\end{example}

\begin{appendix}

\section{Proofs}

We first prove Theorem~\ref{t2} and Theorem~\ref{t3} and then use these 
results to deduce Theorem~\ref{t1}. The proof of Theorem~\ref{t2} is 
modeled after the proof of Theorem 1.1(a) in 
Guionnet and Zeitouni~\citep{Gui00a}. It rests on a slight modification of 
Theorem 6.6 of Talagrand~\citep{Tal96b} that is given as 
Theorem~\ref{mytal} below, and also on Lemma 1.2 from 
Guionnet and Zeitouni~\citep{Gui00a} that is restated as 
Lemma~\ref{tracefunction}, which follows. 

\begin{theorem}\label{mytal}
Fix $m \geq 1$ and $p \geq 1$. Consider a function 
$T : \left[ -1, 1 \right]^{mp} \rightarrow \R$ that is 
quasi-convex\footnote{A real valued function $T$ is said to be 
quasi-convex if all the level sets $\left\{ T \leq a \right\}, a \in \R$, 
are convex.} and Lipschitz with Lipschitz constant $\sigma$. 
Let $Y_1, \dots, Y_m$ be independent $p$-vectors, each taking values in 
$\left[ -1, 1\right]^p$ and consider the random variable 
$T = T(Y_1, \dots, Y_m)$. For each $\epsilon > 0$, we then have
\begin{equation}\label{talineq}
\P \left( |T - \text{med } T| \geq \epsilon \right) \leq 4 
\exp \left( - \frac{1}{p \sigma^2} \frac{\epsilon^2}{16} \right).
\end{equation}
\end{theorem}

The above theorem follows from Theorem 6.1 of Talagrand~\citep{Tal96b} 
by arguing just like in the proof of Theorem 6.6 of Talagrand~\citep{Tal96b}, 
but now using $\left[ -1, 1 \right]^p$ instead of $\left[ -1, 1\right]$. 
When $p = 1$, Theorem~\ref{mytal} reduces to Theorem 6.6 of 
Talagrand~\citep{Tal96b}.

\begin{lemma}\label{tracefunction}
Let $\Acal^n$ denote the set of all real symmetric $n \times n$ matrices 
and let $u : \R \rightarrow \R$ be a fixed function. Let us denote by 
$\Lambda^n_u$ the functional $A \mapsto F_A(u)$ on $\Acal^n$. Then
\begin{list}{\thercnt}{\usecounter{rcnt}\setlength{\leftmargin}{5pt}}
\item \label{trace1} If $u$ is convex, then so is $\Lambda^n_u$. 
\item \label{trace2}
If $u$ is Lipschitz, then so is $\Lambda^n_u$ (when considering $\Acal^n$ 
with the Euclidean norm on $\R^{n(n+1)/2}$ by collecting the entries on and 
above the diagonal). Moreover, the Lipschitz constant of $\Lambda^n_u$ 
satisfies
\begin{equation*}
||\Lambda^n_u||_L \leq \frac{\sqrt{2}}{\sqrt{n}}||u||_L.
\end{equation*}
\end{list}
\end{lemma}

\begin{remark}
For a proof of this lemma, see 
Guionnet and Zeitouni~\citep[Proof of Lemma 1.2]{Gui00a}. A simpler proof 
(along with other similar results) of Lemma~\ref{tracefunction}\ref{trace1} 
can be found in Lieb and Pedersen~\citep{liebtrace}. 
\end{remark}

\begin{proof}[Proof of Theorem~\ref{t2}]
Set $T = F_S(f)$ and let $\Acal^n$ be as in Lemma~\ref{tracefunction}. 
In view of Theorem~\ref{mytal}, it suffices to show that 
$T = T(Y_1, \dots, Y_m)$ is such that the function $T(\cdot)$ is 
quasi-convex and Lipschitz with Lipschitz constant 
$\leq  (2/(n m))^{1/2} C_M ||f||_L$. To this end, we write $T$ as the 
composition $T_2 \circ T_1$, where 
$T_1: \left([-1, 1]^p\right)^m \rightarrow \Acal^n$ and 
$T_2 : \Acal^n \rightarrow \R$ denote the mappings 
$(y_1, \dots, y_m) \mapsto M(y_1, \dots, y_m)/\sqrt{m}$ and 
$A \mapsto F_A(f)$, respectively. By assumption, $T_1$ is linear and Lipschitz 
with $||T_1||_L = C_M/\sqrt{m}$. Also, since $f$ is assumed to be convex and 
Lipschitz, Lemma~\ref{tracefunction} entails that $T_2$ is convex and 
Lipschitz with $||T_2||_L \leq (2/n)^{1/2} ||f||_L$. It follows that 
$T$ is convex (and hence quasi-convex) and Lipschitz with 
$||T||_L \leq (2/(n m))^{1/2}  C_M ||f||_L$. The proof is complete.
\end{proof}

To prove Theorem~\ref{t3}, we recall McDiarmid's bounded difference 
inequality (\citep{Mcd89a}; see also Proposition 12 in~\citep{Bou03a}):

\begin{proposition}\label{bdd}
Consider independent random quantities $Y_1,\dots, Y_m$, and
a (measurable) function $Z = f(Y_1,\dots, Y_m)$.  
For each $i=1,\dots,m$, define $Z_{(i)}$ like $Z$, but with $Y_i$
replaced by an independent copy; that is,
$Z_{(i)} = f(Y_1,\dots, Y_{i-1}, Y_i^\ast,Y_{i+1},\dots, Y_m)$,
where $Y_i^\ast$ is distributed as $Y_i$ and independent of
$Y_1,\dots, Y_m$. If
$$
\big| Z  - Z_{(i)}\big| \quad\leq \quad c_i
$$
holds (almost surely) for each $i=1,\dots,m$, then, for each $\epsilon > 0$, both $\P\left( Z - \E Z \geq \epsilon \right)$ and $\P\left( Z - \E Z \leq - \epsilon \right)$ are bounded by $\exp\left[ - 2 \epsilon^2 / \sum_{i=1}^m c_i^2 \right]$.
\end{proposition}

\begin{proof}[Proof of Theorem~\ref{t3}]
It suffices to prove the second claim. Hence assume that $a$ and $b$, $-\infty \leq a < b \leq \infty$ are such that $\P \left( a < \lambda_1(S) \text{ and } \lambda_n(S) < b \right) = 1$ and that $f : (a, b) \rightarrow \R$ is of bounded variation on $\left( a, b \right)$. We shall now show that
\begin{equation}\label{pt2.1}
|F_S(f) - F_{S_{(i)}}(f)| \qquad \leq \qquad r V_f(a, b)/n
\qquad\qquad (i=1,\dots, m).
\end{equation}
With this, we can use the bounded difference inequality, i.e.,  Proposition~\ref{bdd}, with $Z$, $Z_{(i)}$, and $c_i$  ($1\leq i \leq m$)
replaced by $F_S(f)$, $F_{S_{(i)}}(f)$, and $r V_f(a,b)/n$, respectively, to
obtain \eqref{t3.2}, completing the proof.

To obtain \eqref{pt2.1}, set 
$G(\lambda) = F_S(\lambda) - F_{S_{(i)}}(\lambda)$ and choose $\alpha$ and $\beta$ satisfying $a < \alpha < \min\{ \lambda_1(S), \lambda_1(S_{(i)})\}$ and
$b > \beta > \max\{ \lambda_n(S), \lambda_n(S_{(i)})\}$.
With these choices, we can write $F_S(f) - F_{S_{(i)}}(f)$
as the Riemann-Stieltjes integral $\int_\alpha^\beta f dG$. In particular,
we have
\begin{equation}\nonumber
\Big| F_S(f) - F_{S_{(i)}}(f) \Big| \quad =\quad \Big| \int_\alpha^\beta f d G\Big|
\quad =\quad \Big| \int_\alpha^\beta G d f\Big|
\quad \leq\quad || G||_\infty V_f(a, b),
\end{equation}
where the second equality is obtained through integration by parts
upon noting that $G(\alpha) = G(\beta) =  0$.
By assumption, $||G||_\infty = || F_S - F_{S_{(i)}}||_\infty \leq r/n$,
and \eqref{pt2.1} follows.
\end{proof}

\begin{proof}[Proof of Theorem~\ref{t1}]
Our reasoning is similar to that used in the proof of Corollary~1.8 
of Guionnet and Zeitouni~\citep{Gui00a}. 
Set $\tilde{n} = m + n$ and write $\tilde{M}$ as shorthand for 
$\tilde{n} \times \tilde{n}$ matrix
\[ 
\tilde{M} \qquad=\qquad
\left( \begin{array}{cc}
0_{n \times n} & X'_{n \times m} \\
X_{m \times n} & 0_{m \times m} \end{array} \right).\]
Moreover,
set $\tilde{S} = \tilde{M}/\sqrt{m}$, and write $Y_i$ for the $i$-th row 
of $X$, $1 \leq i \leq m$, i.e., $Y_i = (X_{i,\cdot})'$. 
We view $\tilde{M}$ as a function of 
$Y_1, \dots, Y_m$. Also let $\tilde{f}(x) = f\left( x^2 \right)$. 

It is easy to check that
\begin{equation*}
F_{\tilde{S}}( \tilde{f} ) = \frac{2n}{\tilde{n}} F_S(f) + 
\frac{m - n}{\tilde{n}} f(0),
\end{equation*}
and hence
\begin{equation*}
\P \left( |F_S(f) - \mu| > \epsilon \right) \quad = \quad 
\P \left( |F_{\tilde{S}}(\tilde{f}) - \tilde{\mu}| > 
\frac{2n}{\tilde{n}}\epsilon \right),
\end{equation*}
where $\mu$ ($\tilde{\mu}$) can be either $\E F_S(f)$ 
($\E F_{\tilde{S}}(\tilde{f})$) or $\text{med } F_S(f)$ 
($\text{med }F_{\tilde{S}}(\tilde{f})$).

To prove~$\ref{t1.i}$, it suffices to note that Theorem~\ref{t2} applies with 
$\tilde{M}$, $\tilde{S}$, $\tilde{n}$, $n$,
$\tilde{f}$, and $1$  replacing $M$, $S$, $n$, $p$,
$f$, and $C_M$, respectively. 
Using Theorem~\ref{t2} with these replacements and with 
$\frac{2n}{\tilde{n}}\epsilon$ replacing $\epsilon$, we see that the left 
hand side of~\eqref{t1.i.1} is bounded as claimed.

To prove~\ref{t1.ii}, we first note that 
$||F_{\tilde{S}} - F_{\tilde{S}^{(i)}}||_{\infty} \leq 2/\tilde{n}$ in view of 
Lemma~\ref{bai} (where $\tilde{S}^{(i)}$ is defined as $\tilde{S}$ but with 
$Y_i$ replaced by an independent copy). Also, note that $\tilde{f}$ is of 
bounded variation on $\R$ with $V_{\tilde{f}}(\R) \leq V_f(\R)$. Hence, 
Theorem~\ref{t3} applies with $\tilde{M}, \tilde{S}, \tilde{n}, r$ and 
$\tilde{f}$ replacing $M, S, n, 2$ and $f$ respectively and~\eqref{t1.ii.1} 
follows after elementary simplifications.
\end{proof}

\end{appendix}

\subsection*{Acknowledgment}
We sincerely thank Professor Ofer Zeitouni for valuable comments that greatly 
helped us to improve an earlier version of the paper.

\bibliographystyle{acmtrans-ims}

\begin{thebibliography}{}
\ifx \url   \undefined \def \url#1{#1}   \fi

\bibitem{Bai08a}
\textsc{Bai, Z.} \textsc{and} \textsc{Zhou, W.} (2008).
\newblock Large sample covariance matrices without independence sturctures in
  columns.
\newblock \emph{Statist. Sinica\/}~\emph{{\bf 18}}, 425--442.

\bibitem{Bai99a}
\textsc{Bai, Z.~D.} (1999).
\newblock Methodologies in spectral analysis of large-dimensional random
  matrices, a review.
\newblock \emph{Statist. Sinica\/}~\textbf{9}, 611--677.
\newblock With comments by G. J.\ Rodgers and Jack W.\ Silverstein; and a
  rejoinder by the author.

\bibitem{Bou03a}
\textsc{Boucheron, S.}, \textsc{Lugosi, G.}, \textsc{and} \textsc{Massart, P.}
  (2003).
\newblock Concentration inequalities using the entropy method.
\newblock \emph{Ann. Probab.\/}~\textbf{31}, 1583--1614.

\bibitem{Bou98a}
\textsc{Boutet~de Monvel, A.} \textsc{and} \textsc{Khorunzhy, A.} (1998).
\newblock Limit theorems for random matrices.
\newblock \emph{Markov Process. Related Fields\/}~\textbf{4}, 175--197.

\bibitem{Che52a}
\textsc{Chernoff, H.} (1952).
\newblock A measure of asymptotic efficiency for tests of a hypothesis based on
  the sum of observations.
\newblock \emph{Ann. Math. Stat.\/}~\emph{{\bf 23}}, 493--507.

\bibitem{Got04a}
\textsc{G{\"o}tze, F.} \textsc{and} \textsc{Tikhomirov, A.} (2004).
\newblock Limit theorems for spectra of positive random matrices under
  dependence.
\newblock \emph{Zap. Nauchn. Sem. S.-Peterburg. Otdel. Mat. Inst. Steklov.
  (POMI)\/}~\textbf{311},~Veroyatn. i Stat. 7, 92--123, 299.

\bibitem{Gui04a}
\textsc{Guionnet, A.} (2004).
\newblock Large deviations and stochastic calculus for large random matrices.
\newblock \emph{Probab. Surv.\/}~\emph{{\bf 1}}, 72--172 (electronic).

\bibitem{Gui00a}
\textsc{Guionnet, A.} \textsc{and} \textsc{Zeitouni, O.} (2000).
\newblock Concentration of the spectral measure for large matrices.
\newblock \emph{Electron. Comm. Probab.\/}~\emph{{\bf 5}}, 119--136 (electronic).

\bibitem{Hou08a}
\textsc{Houdre, C.} \textsc{and} \textsc{Xu, H.} (2008).
\newblock Concentration of the spectral measure for large random matrices with
  stable entries.
\newblock \emph{Electron. J. Probab.\/}~\emph{{\bf 5}}, 107--134.

\bibitem{Lan93a}
\textsc{Lang, S.} (1993).
\newblock \emph{Real and Functional Analysis}, Third ed. Graduate Texts in
  Mathematics, Vol. \emph{{\bf 142}}.
\newblock Springer-Verlag, New York.

\bibitem{liebtrace}
\textsc{Lieb, E.} \textsc{and} \textsc{Pedersen, G.} (2002).
\newblock Convex multivariable trace functions.
\newblock \emph{Rev. Math. Phys.\/}~\emph{{\bf 14}}, 631--648.

\bibitem{Mcd89a}
\textsc{McDiarmid, C.} (1989).
\newblock On the method of bounded differences.
\newblock In \emph{Surveys in Combinatorics, 1989 (Norwich, 1989)}. London
  Math. Soc. Lecture Note Ser., Vol. \emph{{\bf 141}}. Cambridge Univ. Press,
  Cambridge, 148--188.

\bibitem{Men06a}
\textsc{Mendelson, S.} \textsc{and} \textsc{Pajor, A.} (2006).
\newblock On singular values of matrices with independent rows.
\newblock \emph{Bernoulli\/}~\emph{{\bf 12}}, 761--773.

\bibitem{Tal96b}
\textsc{Talagrand, M.} (1996).
\newblock A new look at independence.
\newblock \emph{Ann. Probab.\/}~\emph{{\bf 24}}, 1--34.

\bibitem{Wal23a}
\textsc{Walsh, J.~L.} (1923).
\newblock A closed set of normal orthogonal functions.
\newblock \emph{Amer. J. Math.\/}~\emph{{\bf 45}}, 5--24.

\end{thebibliography}

\end{document}